\documentclass[12pt,twoside]{amsart}
\usepackage{amsfonts}
\usepackage{amsmath}
\usepackage{amssymb}
\usepackage{mathrsfs}
\newtheorem{theo}{Theorem}
\newtheorem{cor}{Corollary}
\newtheorem{lem}{Lemma}
\newtheorem{prop}{Proposition}
\theoremstyle{definition}
\newtheorem{defn}{Definition}
\theoremstyle{remark}
\newtheorem{rem}{\bf Remark\/}
\newtheorem{prob}{\bf  Problem\/}

\numberwithin{equation}{section}
\def\1{{\mathchoice {\rm 1\mskip-4mu l} {\rm 1\mskip-4mu l}{\rm 1\mskip-4.5mu l} {\rm 1\mskip-5mu l}}}
\newcommand{\ds}{\displaystyle}
\title{Lelong numbers of $m-$subharmonic functions}

\author[A. Benali]{Amel Benali}
\email{amelbenali3010@gmail.com}
\address{Laboratory of mathematics and applications \\ Faculty of sciences of Gab\`es \\ University of Gab\`es \\ 6072 Gab\`es Tunisia.}
\author[N. Ghiloufi]{Noureddine Ghiloufi}
\email{noureddine.ghiloufi@fsg.rnu.tn}
\address{Department of mathematics\\College of science\\P.O. box 400 King Faisal University \\  Al-Ahsaa, 31982\\ Kingdom of Saudi Arabia.}
\begin{document}

\begin{abstract}
    In this paper we study the existence of Lelong numbers of $m-$subharmonic currents of bidimension $(p,p)$ on an open subset of $\Bbb C^n$, when $m+p\geq n$. In the special case of $m-$subharmonic function $\varphi$, we give a relationship between the Lelong numbers of $dd^c\varphi$ and the mean values of $\varphi$ on spheres or balls. As an application we study the integrability exponent of $\varphi$. We express the integrability exponent  of $\varphi$  in terms of volume of sub-level sets of $\varphi$  and we give a  link between this exponent and  its Lelong number.
\end{abstract}
\maketitle
%\tableofcontents

\section{Introduction}
In complex analysis and geometry, the notion of Lelong numbers of  positive currents has many applications. A famous result due to Siu \cite{Si}, proves that if $T$ is a positive closed current of bidimension $(p,p)$ on an open set $\Omega$ of $\mathbb C^n$, then the level subset $E_T(c)$ of points $z$ where the Lelong number $\nu_T(z)$ of $T$ at $z$ is greater than or equal to $c$ is an analytic set of dimension less than or equal to $p$ for any real $c>0$. In a particular case, if  $u$ is a plurisubharmonic function  on $\Omega $ then the Lelong number $\nu_u(a)$  of $u$ (of the current $dd^cu$)  at a point $a\in \Omega$  characterizes the complex singularity exponent $c_a(u)$ of $u$ at $a$.  In fact, Skoda \cite{Sk} stated the following inequalities:
$$\ds\frac{1}{\nu_{u}(a)}\leq c_a(u)\leq \frac{n}{\nu_{u}(a)}.$$
This result was enhanced by Demailly and Pham \cite{De-Ph}, who showed a sharp relationship between this exponent and  the Lelong numbers of  currents $(dd^c\varphi)^j$ at $a$ for $1\leq j\leq n$.\\

In 2005, Blocki \cite{Bl} posed a problem about the integrability exponents of $m-$subharmonic functions, he conjectured that every $m-$subharmonic function belongs to $L^q_{loc}(\Omega)$ for any $q<\frac{nm}{n-m}$ ie. the integrability exponent (the supremum of $q$) is greater than or equal to $\frac{nm}{n-m}$. This problem may be similar, but different, to the study of the complex singularity exponents of plurisubharmonic functions. To study this similar problem, we need a suitable definition of Lelong numbers of $m-$subharmonic functions. In 2016, Wan and Wang \cite{Wa-Wa} give the definition of the Lelong number of an $m-$subharmonic function $\psi$  as the Lelong number of the current $dd^c\psi$:
$$\nu_\psi(x)=\ds\lim_{r\to0^+}\frac{1}{r^{\frac{2n}{m}(m-1)}}\int_{\mathbb{B}(x,r)}dd^c\psi\wedge\beta^{n-1}.$$
Our aim is to find a link between the integrability exponents of $\psi$ and its Lelong numbers. To reach this aim, we need to know more properties of Lelong numbers.\\

The paper is organized as follows: in section $2$, we introduce the basic concepts which will be employed in the rest of this paper. Indeed, we recall the notions of $m-$positive currents, $m-$subharmonic functions and their Lelong numbers with some properties.\\

Section $3$ is devoted to the study of the existence of Lelong numbers of $m-$positive currents. We start by proving, the main tool in this study, the Lelong-Jensen formula. By this formula we conclude that, for an $m-$positive $m-$subharmonic current $T$  on $\Omega$, the Lelong function $\nu_T(a,\centerdot)$ associated to $T$ at $a\in\Omega$  is increasing on $]0,d(a,\partial\Omega)[$. Thus, its limit at zero $\nu_T(a)$ exists.
 Moreover, we study the case of an $m-$negative $m-$subharmonic current $S$, we show that $\nu_S(a)$ exists with the assumption that $t\mapsto t^{\frac{-2n}{m}+1}\nu_{dd^cS}(a,t)$ is integrable in a neighborhood of $0$.\\
 
In section $4$, we give a relationship between the Lelong number of an $m-$subharmonic function $\psi$ at $a\in\Omega$ and its mean values on spheres and balls. We conclude then  that the map $z\longmapsto\nu_\psi(z)$ is upper semi-continuous on $\Omega$.\\

Finally, as an application, we study in the last section  the integrability exponent $\imath_K(\psi)$ of an $m-$subharmonic function $\psi$ on a compact subset  $K$ of $\Omega$. We prove that if $\psi<0$ on a neighborhood of $K$ then we have
$$\imath_K(\psi)=\ds\sup\left\{\alpha>0; \exists\ C_{\alpha}>0, \forall\; t<0 \ V(\{\psi<t\})\cap K )\leq\frac{C_{\alpha}}{|t|^{\alpha}}\right\}.$$
At the end, we show that if $\nu_\varphi(a)>0$  then $$\frac{n}{n-m}\leq\imath_a(\psi)\leq\frac{nm}{n-m}.$$ In particular if the Blocki conjecture is true then we have $\imath_a(\psi)=\frac{nm}{n-m}$. We claim that this equality is true for $m=1$ and $\nu_\psi(a)>0$ and this result can be viewed as a  partial answer for the Blocki conjecture.

\section{Preliminaries}

Throughout this paper $\Omega$ is an open set of $\Bbb C^n$ and $m$ is an integer such that $1\leq m< n$. We use the operators $d=\partial+\overline{\partial}$ and $d^c=\frac{i}{4\pi}(\overline{\partial}-\partial)$ in order to have  $dd^c=\frac{i}{2\pi}\partial\overline{\partial}$. We set $\beta=dd^c|z|^2$ and $$\phi_m(r)=-\frac{1}{(\frac{n}{m}-1)r^{2(\frac{n}{m}-1)}}.$$ For $x\in\Bbb C^n$, $r>0$ and $0<r_1<r_2$ we set
  $$ \Bbb B(x,r_1,r_2)=\{z\in\Bbb C^n;\ r_1<|z-x|<r_2\}$$ and $$\Bbb B(x,r)=\{z\in\Bbb C^n;\ |z-x|<r\}.$$
In this part we recall some definitions of $m-$positivity cited by Dhouib-Elkhadhra in \cite{Dh-El}.
\begin{defn}${\\}$
	\begin{enumerate}
		\item A $(1,1)-$form $\alpha$ on $\Omega$ is said to be $m-$positive if $\alpha^j\wedge \beta^{n-j}\geq 0$ (in sens of currents) for every $1\leq j\leq m$.
		\item A $(p,p)-$form $\alpha$ on $\Omega$ is strongly $m-$positive if $$\alpha=\sum_{k=1}^N a_k\alpha_{1,k}\wedge \dots\wedge \alpha_{p,k}$$
		where $N={n\choose p}$ and $\alpha_{1,k},\dots \alpha_{p,k}$ are $m-$positive  $(1,1)-$forms and $a_k\geq0$ for every $k$.
		\item A current  $T$ of bidimension $(p,p)$ on $\Omega$ with  $m+p\geq n$  is said to be $m-$positive if $\langle T\wedge \beta^{n-m},\alpha\rangle\geq0$ for every strongly $m-$positive $(m+p-n,m+p-n)-$ test form $\alpha$ on $\Omega$.
		\item A function $\varphi:\Omega \rightarrow \mathbb{R}\cup\{-\infty\}$ is said to be  $m-$subharmonic ($m-$sh for short) if it is subharmonic and $dd^c\varphi$ is an $m-$positive current on $\Omega$.\\
We set $\mathcal{SH}_m(\Omega)$ the set of $m-$subharmonic functions on $\Omega$.
 	\end{enumerate}	
\end{defn}

 In general, if $T$ is an $m-$positive current of bidimension $(p,p)$ on $\Omega$ and $a\in\Omega$, the $m-$Lelong function of $T$ at $a$ is defined by:
$$\nu_T(a,r):= \ds\frac{1}{r^{\frac{2n}{m}(m+p-n)}}\ds\int_{\mathbb{B}(a,r)}T\wedge \beta^{p}.$$
for $r<d(a,\partial\Omega)$.
The Lelong number of $T$ at $a$, when it exists, is $$\nu_T(a)=\ds\lim_{r\to 0^+}\nu_T(a,r).$$

Here, we give a short list of the most basic properties of $m-sh$ functions:
\begin{prop}
Let $\Omega\subset\mathbb{C}^n$ be a domain.
\begin{enumerate}
  \item If $\varphi\in \mathcal{C}^2(\Omega)$, then $\varphi$ is  $m-$sh if and only if   $$(dd^c \varphi)^k\wedge \beta^{n-k}\geq0$$
for $k=1,2,...,m$, in the sens of currents.
 \item $\mathcal{PSH}(\Omega)=\mathcal{SH}_n(\Omega)\subsetneq \mathcal{SH}_{n-1}(\Omega)\subsetneq...\subsetneq \mathcal{SH}_1(\Omega)=\mathcal{SH}(\Omega)$.
  \item $\mathcal{SH}_m(\Omega)$ is a convex cone.
  \item If $\varphi$ is $m-$sh  and $\gamma:\mathbb{R}\rightarrow\mathbb{R}$ is a $\mathcal{C}^2-$smooth convex, increasing function then $\gamma\circ \varphi$ is also $m-$sh.
  \item The standard regularization $\varphi*\rho_{\varepsilon}$ of an $m-sh$ function is again $m-sh$.
  \item The limit of a uniformly converging or decreasing sequence of $m-$sh functions is either $m-$sh or identically equal to $-\infty$.
   \end{enumerate}
\end{prop}
Now we recall some classes of $m-$sh functions on $\Omega$, in relation with the definition of the complex Hessian operator called Cegrell classes (see  \cite{Ha-Za} for more details):
\begin{itemize}
\item $\mathcal{E}_{0,m}(\Omega)$ is the convex cone of bounded negative $m-$sh function $\varphi$ on $\Omega$ such that $$\ds\lim_{z\rightarrow\partial\Omega}\varphi(z)=0\quad \hbox{and}\quad\int_{\Omega}(dd^c\varphi)^m\wedge\beta^{n-m}<+\infty.$$
 \item $\mathcal{F}_m(\Omega)$ is the class of negative $m-$sh functions on $\Omega$ such that there exists a sequence $(\varphi_j)_j$ in $ \mathcal{E}_{0,m}(\Omega)$ that decreases to $\varphi$ and $$\sup_j\int_\Omega(dd^c\varphi_j)^m\wedge\beta^{n-m}<+\infty.$$
\item We denote by $ \mathcal{E}_m(\Omega)$ the subclass of  negative $m-$sh functions  on $\Omega$  that coincides locally with elements of $\mathcal{F}_m$.
\end{itemize}

In the next, we introduce some properties that will employed in the sequel:\\

\begin{prop}(See \cite{Ha-Za})
\begin{itemize}
  \item $\mathcal{E}_{0,m}(\Omega)\subset \mathcal{F}_m(\Omega)\subset \mathcal{E}_m(\Omega)$.
  \item If $ \varphi\in\mathcal{E}_{0,m}(\Omega)$ and $\psi\in \mathcal{SH}_m^-(\Omega)$, then $\max(\varphi,\psi)\in\mathcal{E}_{0,m}(\Omega)$.
  \item If $\varphi\in\mathcal{F}_m(\Omega)$ then $\int_{\Omega}(dd^c\varphi)^p\wedge\beta^{n-p}<+\infty$.
\end{itemize}
\end{prop}

\begin{lem}(See \cite{Ha-Za})\label{lem1}
Suppose that $\varphi_1,\dots,\varphi_2\in \mathcal{F}_m(\Omega)$ and $h\in \mathcal{E}_{0,m}(\Omega)$. Then we have
$$\begin{array}{l}
\ds\int_{\Omega}-hdd^c\varphi_1\wedge\dots\wedge dd^c\varphi_m\wedge \beta^{n-m}\\
\leq\ds \left(\int_{\Omega}-h (dd^c\varphi_1)^m\wedge\beta^{n-m}\right)^{\frac{1}{m}}\dots\left(\int_{\Omega}-h(dd^c\varphi_m)^m\wedge\beta^{n-m}\right)^{ \frac{1}{m} }.
\end{array}$$
\end{lem}

 \section{Lelong numbers of $m-$subharmonic currents}
The aim of this part is to prove the existence of the Lelong number of $m-$subharmonic currents. The main tool is the Lelong-Jensen formula.
%\subsection{Lelong-Jensen formula}
\begin{prop}(Lelong-Jensen formula)
	Let $T$ be a current of bidimension $(p,p)$ on $\Omega$  such that $T$ and $dd^cT$ are of zero order on $\Omega$. Then for every $a\in\Omega$ and $0<r_1<r_2<d(a,\partial\Omega)$, we have
	$$\begin{array}{lcl}
		A(r_1,r_2)&:=&\nu_T(a,r_2)-\nu_T(a,r_1)\\
		&=&\ds\frac{1}{r_2^{\frac{2n}{m}(m+p-n)}}\int_{\mathbb{B}(a,r_2)}T\wedge\beta^p- \frac{1}{r_1^{\frac{2n}{m}(m+p-n)}}\int_{\mathbb{B}(a,r_1)}T\wedge\beta^p\\
		&=&\ds \int_{r_1}^{r_2}\left(\frac{1}{t^{\frac{2n}{m}(m+p-n)}} -\frac{1}{r_2^{\frac{2n}{m}(m+p-n)}}\right)2tdt\int_{\mathbb{B}(a,t)}dd^cT\wedge\beta^{p-1}\\
		&&\ds +\int_0^{r_1}\left(\frac{1}{r_1^{\frac{2n}{m}(m+p-n)}} -\frac{1}{r_2^{\frac{2n}{m}(m+p-n)}}\right)2tdt\int_{\mathbb{B}(a,t)}dd^cT\wedge\beta^{p-1}\\
		& &\ds +\int_{\mathbb{B}(a,r_1,r_2)}T(\xi)\wedge\beta(\xi)^{n-m}\wedge(dd^c\widetilde{\phi}_m(\xi-a))^{m+p-n}.
	\end{array}$$
    where $\widetilde{\phi}_m(\zeta)=\phi_m(|\zeta|)$.
\end{prop}
\begin{proof}
	Without loss of generality, we can assume that $a=0$. We use $\mathbb{B}(r)$ and $\mathbb{B}(r_1,r_2)$ instead of $\mathbb{B}(0,r)$ and $\mathbb{B}(0,r_1,r_2)$. We set $\mathbb{S}(r)=\partial\mathbb{B}(r)$.\\
Suppose first that $T$ is of class $\mathcal{C}^2$. Then thanks to Stokes formula, we have
	\begin{equation}\label{lj1}
\begin{array}{ll}
    &\ds \int_{r_1}^{r_2}\frac{2tdt}{t^{\frac{2n}{m}(m+p-n)}} \int_{\mathbb{B}(t)}dd^cT\wedge\beta^{p-1}\\
	=&\ds \int_{r_1}^{r_2}\frac{2tdt}{t^{\frac{2n}{m}(m+p-n)}}\int_{\mathbb{S}(t)}d^cT\wedge\beta^{p-1}\\
	=&\ds\int_{r_1}^{r_2}2tdt\int_{\mathbb{S}(t)}d^cT\wedge(dd^c\widetilde{\phi}_m)^{m+p-n}\wedge\beta^{n-m-1}\\
	=&\ds \int_{\mathbb{B}(r_1,r_2)}d|z|^2\wedge d^cT\wedge(dd^c\widetilde{\phi}_m)^{m+p-n}\wedge\beta^{n-m-1}\\
	=&\ds \int_{\mathbb{B}(r_1,r_2)}dT\wedge d^c|z|^2\wedge(dd^c\widetilde{\phi}_m)^{m+p-n}\wedge\beta^{n-m-1}\\
	=&\ds \int_{\mathbb{B}(r_1,r_2)}d\left(T\wedge d^c|z|^2\wedge(dd^c\widetilde{\phi}_m)^{m+p-n}\wedge\beta^{n-m-1}\right)\\
    &-\ds\int_{\mathbb{B}(r_1,r_2)}T\wedge(dd^c\widetilde{\phi}_m)^{m+p-n}\wedge\beta^{n-m}\\	
	=&\ds \int_{\mathbb{S}(r_2)}T\wedge d^c|z|^2\wedge(dd^c\widetilde{\phi}_m)^{m+p-n}\wedge\beta^{n-m-1}\\
	&-\ds \int_{\mathbb{S}(r_1)}T\wedge d^c|z|^2\wedge(dd^c\widetilde{\phi}_m)^{m+p-n}\wedge\beta^{n-m-1}\\
	&-\ds\int_{\mathbb{B}(r_1,r_2)}T\wedge(dd^c\widetilde{\phi}_m)^{m+p-n}\wedge\beta^{n-m}	
	\end{array}
\end{equation}

	A simple computation shows that
	\begin{equation}\label{lj2}
	\begin{array}{ll}
	\ds \int_{\mathbb{S}(r)}T\wedge d^c|z|^2\wedge(dd^c\widetilde{\phi}_m)^{m+p-n}\wedge\beta^{n-m-1}\\
	= \ds \frac{1}{r^{\frac{2n}{m}(m+p-n)}}\int_{\mathbb{S}(r)}T\wedge d^c|z|^2\wedge\beta^{p-1}\\ =\ds\frac{1}{r^{\frac{2n}{m}(m+p-n)}}\int_0^r2tdt\int_{\mathbb{B}(t)}dd^cT\wedge\beta^{p-1}+\frac{1}{r^{\frac{2n}{m}(m+p-n)}}\int_{\mathbb{B}(t)}T\wedge\beta^p.
	\end{array}
	\end{equation}
	The result follows from Equalities (\ref{lj1}) and (\ref{lj2}).\\
	If $T$ is not of class $\mathcal C^2$, we consider the set
	$$E_T:=\left\{r>0;\ ||T||(\mathbb{S}(r))\not=0 \hbox{ or } ||dd^cT||(\mathbb{S}(r))\not=0\right\}.$$
	As $T$ and $dd^cT$ are of zero orders, then $E_T$ is at least countable. Let $(\rho_\varepsilon)_\varepsilon$ be a regularizing kernel and $r\in\mathbb{R}\smallsetminus E_T$, then we have
	$$\lim_{\varepsilon\to 0}\int_{\mathbb{B}(r)}T*\rho_\varepsilon\wedge\beta^p =\lim_{\varepsilon\to 0}\int_{\mathbb{C}^n}\1_{\mathbb{B}(r)}T*\rho_\varepsilon\wedge\beta^p=\int_{\mathbb{B}(r)}T\wedge\beta^p.$$
	It follows that if $0<r_1<r_2$ are two values outside $E_T$ then the result is checked by regularization. If $r_1$ or $r_2$ is in $E_T$, it suffices to take two sequences $(r_{1,j})_j$ and  $(r_{2,j})_j$ in $\mathbb{R}\smallsetminus E_T$ which tend respectively to $r_1$ and $r_2$ and apply the previous step with $r_{1,j}$ and  $r_{2,j}$. The result follows by passing to the limit when $j\to+\infty$.	
\end{proof}
As a consequence, we have:
\begin{theo}\label{th1}
	Let $T$ be a current of bidimension $(p,p)$ on $\Omega$. Assume that $T$ is $m-$positive  and $dd^cT\wedge\beta^{p-1}$ is a positive measure  on $\Omega$ with $m+p\geq n$. Then the Lelong number of $T$ exists at every point of $\Omega$.
\end{theo}
This result is due to Wan and Wang for $T=dd^c\varphi$ where $\varphi$ is an $m-$sh function.
\begin{proof}
Since $T$ is an $m-$positive  current on $\Omega$ then $\nu_T(a,\centerdot)$ is positive. Moreover, thanks to Lelong-Jensen formula, $\nu_T(a,\centerdot)$ is increasing. It follows that its limit $\nu_T(a)$ when $r$ tends to $0$ exists.
\end{proof}

% \subsection{Case of $m-$negative $m-$subharmonic currents}
The case of $m-$negative $m-$subharmonic currents is so different to the previous case (of $m-$positive $m-$subharmonic currents). Indeed, let $T_0:=\widetilde{\phi}_m(dd^c\widetilde{\phi}_m)^{m-1}$; it is not hard to see that $T_0$ is $m-$negative $m-$subharmonic current  of bidimension $(n-m+1,n-m+1)$ on $\Bbb C^n$ with $dd^cT_0\wedge\beta^{n-m}=\delta_0$ and $\nu_{T_0}(r):=\nu_{T_0}(0,r)=\frac{c_n}{r^{2(\frac{n}{m}-1)}}$ for some constant $c_n<0$. Thus the Lelong number of $T_0$ at $0$ doesn't exist.\\
It follows that it is legitimate to impose a condition on an $m-$negative $m-$subharmonic current to ensure the existence of its Lelong number, this will be the aim of Theorem \ref{th3}. But before giving  such a  condition, we may study the local behavior of the Lelong function associated to such a current.
\begin{lem}
Let $T$ be an $m$-negative $m-$subharmonic current of bidimension $(p,p)$ on $\Omega$ with $m+p-1\geq n$. Then for every $a\in\Omega$ and $0<r_0<d(a,\partial\Omega)$, there exists $c_0<0$ such that for any $0<r\leq r_0$ we have:
$$\nu_T(a,r)\geq\frac{\nu_{dd^cT}(a,r_0)}{1-\frac{n}{m}}r^{2(1-\frac{n}{m})}+c_0$$
\end{lem}
\begin{proof}
Without loss of generality, we can assume that $a=0$. For $r\leq r_0$ we set:
$$\Upsilon_T(r)=\nu_T(r)-\frac{\nu_{dd^cT}(r_0)}{1-\frac{n}{m}}r^{2(1-\frac{n}{m})}$$
 Thanks to Lelong-Jensen formula, for any $r_1<r_2\leq r_0$, one has:
$$\begin{array}{ll}
 & \Upsilon_T(r_2)-\Upsilon_T(r_1)\\
 =&\ds \nu_T(r_2)-\nu_T(r_1)-\frac{\nu_{dd^cT}(r_0)}{1-\frac{n}{m}}\left(r_2^{2(1-\frac{n}{m})}-r_1^{2(1-\frac{n}{m})}\right) \\
 =&\ds 2\int_{r_1}^{r_2}\left(\frac{1}{t^{\frac{2n}{m}(m+p-n)}}-\frac{1}{r_2^{\frac{2n}{m}(m+p-n)}}\right) t^{\frac{2n}{m}(m+p-1-n)+1}\nu_{dd^cT}(t)dt\\
   &+\ds 2\int_0^{r_1}\left(\frac{1}{r_1^{\frac{2n}{m}(m+p-n)}}-\frac{1}{r_2^{\frac{2n}{m}(m+p-n)}}\right) t^{\frac{2n}{m}(m+p-1-n)+1}\nu_{dd^cT}(t)dt\\
   &+\ds\int_{\Bbb B(r_1,r_2)}T\wedge\beta^{n-m}\wedge(dd^c\widetilde{\phi}_m)^{m+p-n}-\frac{\nu_{dd^cT}(r_0)}{1-\frac{n}{m}}\left(r_2^{2(1-\frac{n}{m})}-r_1^{2(1-\frac{n}{m})}\right)\\
   =&\ds\int_{\Bbb B(r_1,r_2)}T\wedge\beta^{n-m}\wedge(dd^c\widetilde{\phi}_m)^{m+p-n}-\frac{\nu_{dd^cT}(r_0)}{1-\frac{n}{m}}\left(r_2^{2(1-\frac{n}{m})}-r_1^{2(1-\frac{n}{m})}\right)\\
    &+\ds 2\int_{r_1}^{r_2}t^{1-\frac{2n}{m}}\nu_{dd^cT}(t)dt-2\ds\int_0^{r_2}\frac{t^{\frac{2n}{m}(m+p-1-n)+1}}{r_2^{\frac{2n}{m}(m+p-n)}}\nu_{dd^cT}(t)dt\\
   &+2\ds\int_0^{r_1}\frac{t^{\frac{2n}{m}(m+p-1-n)}}{r_1^{\frac{2n}{m}(m+p-n)+1}}\nu_{dd^cT}(t)dt\\
   =&\ds\int_{\Bbb B(r_1,r_2)}T\wedge\beta^{n-m}\wedge(dd^c\widetilde{\phi}_m)^{m+p-n}+2\ds\int_{r_1}^{r_2}(\nu_{dd^cT}(t)-\nu_{dd^cT}(r_0))t^{-\frac{2n}{m}+1}dt\\
    &-2\ds\int_0^{r_2}\frac{t^{\frac{2n}{m}(m+p-1-n)+1}}{r_2^{\frac{2n}{m}(m+p-n)}}\nu_{dd^cT}(t)dt +2\ds\int_0^{r_1}\frac{t^{\frac{2n}{m}(m+p-1-n)+1}}{r_1^{\frac{2n}{m}(m+p-n)}}\nu_{dd^cT}(t)dt\leq 0.
\end{array}$$
Indeed, since $T$ is $m-$negative then $T\wedge\beta^{n-m}\wedge(dd^c\widetilde{\phi}_m)^{m+p-n}$ is a negative measure so
$$\int_{\Bbb B(r_1,r_2)}T\wedge\beta^{n-m}\wedge(dd^c\widetilde{\phi}_m)^{m+p-n}\leq 0.$$
Moreover, as $dd^cT$ is an $m-$positive closed current, then thanks to Theorem \ref{th1}, $\nu_{dd^cT}$ is an  increasing function on $]0,r_0]$. Hence we have
$$\int_{r_1}^{r_2}(\nu_{dd^cT}(t)-\nu_{dd^cT}(r_0))t^{-\frac{2n}{m}+1}dt\leq 0.$$
Furthermore, if we set $$f(r)=-\frac{1}{r^{\frac{2n}{m}(m+p-n)}}\ds\int_0^rt^{\frac{2n}{m}(m+p-1-n)+1}\nu_{dd^cT}(t)dt$$
then $f$ is an absolutely continuous function on $]0,r_0]$ and satisfies:
$$\begin{array}{lcl}
f'(r)&=&\ds\frac{\frac{2n}{m}(m+p-n)}{r^{\frac{2n}{m}(m+p-n)+1}}\ds\int_0^r t^{\frac{2n}{m}(m+p-1-n)+1}\nu_{dd^cT}(t)dt-r^{-\frac{2n}{m}+1}\nu_{dd^cT}(r)\\
     &\leq&\ds \frac{\frac{2n}{m}(m+p-n)}{\frac{2n}{m}(m+p-n)+2} r^{1-\frac{2n}{m}}\nu_{dd^cT}(r)-r^{1-\frac{2n}{m}}\nu_{dd^cT}(r)\leq0
\end{array}$$
for almost every $0<r<r_0$. As a consequence, $\Upsilon_T$ is a decreasing function on $]0,r_0]$, thus $\Upsilon_T(r)\geq\Upsilon_T(r_0)$  for  every $0<r\leq r_0$. We conclude that we have for  every $0<r\leq r_0$,
$$\nu_T(r)\geq\Upsilon_T(r_0)+\nu_{dd^cT}(r_0)\frac{r^{2(1-\frac{n}{m})}}{1-\frac{n}{m}}.$$
The result follows by choosing for example $c_0=\min(0,\Upsilon_T(r_0))$.
 \end{proof}
\begin{theo}\label{th3}
Let $T$ be an $m-$negative  $m-$subharmonic current of bidimension $(p,p)$ on $\Omega$. Assume that $t\longmapsto t^{-\frac{2n}{m}+1}\nu_{dd^cT}(z_0,t)$ is integrable in  neighborhood of $0$ for $z_0\in\Omega$. Then the Lelong number $\nu_T(z_0)$ of $T$ at $z_0$ exists.
\end{theo}
\begin{proof}
It suffices to prove the result with $z_0=0$. For every $0<r\leq r_0<d(0,\partial\Omega)$, we set:
$$g(r)= \nu_T(r)+2\ds\int_0^r\left(\frac{t^{\frac{2n}{m}(m+p-n)}}{r^{\frac{2n}{m}(m+p-n)}} -1\right)t^{-\frac{2n}{m}+1}\nu_{dd^cT}(t)dt.$$
    The assumption  implies that the function  $g$ is well defined and negative on $]0,r_0[$.
   Moreover, using the Lelong-Jensen formula, one can prove that for any $0<r_1<r_2\leq r_0$,
   $$g(r_2)-g(r_1)=\ds\int_{\Bbb B(r_1,r_2)}T\wedge\beta^{n-m}\wedge(dd^c\widetilde{\phi}_m)^{m+p-n}\leq0.$$
   It follows that $g$ is a negative decreasing function on $]0,r_0]$, which gives the existence of the limit $$\lim_{r\to 0^+}g(r)=\lim_{r\to 0^+}\nu_T(r)$$ because $t\longmapsto t^{-\frac{2n}{m}+1}\nu_{dd^cT}(t)$ is integrable in  neighborhood of $0$ and $((t/r)^{\frac{2n}{m}(m+p-n)}-1)$ is uniformly bounded.
\end{proof}

\section{Lelong numbers of $m-$subharmonic functions}
 In this particular case we give a new expression of Lelong number of $dd^c\varphi$ using the mean values of the $m-$sh function $\varphi$ on spheres and balls analogous to the case of plurisubharmonic functions; for this reason we set, as usual,
$$\nu_\varphi(a)=\lim_{r\to0^+}\nu_\varphi(a,r):= \ds\lim_{r\to0^+}\frac{1}{r^{\frac{2n}{m}(m-1)}}\ds\int_{\mathbb{B}(a,r)}dd^c\varphi\wedge \beta^{n-1}$$ this number and we consider the mean values of $\varphi$ over the ball and  the sphere respectively:
$$\begin{array}{lcl}
\Lambda(\varphi,a,r)&=&\ds\frac{n!}{\pi^n r^{2n}}\int_{\mathbb{B}(a,r)}\varphi(x)dV(x)\\
\lambda(\varphi,a,r)&=&\ds\frac{(n-1)!}{2\pi^nr^{2n-1}}\int_{\mathbb{S}(a,r)}\varphi(x)d\sigma(x).
\end{array}$$

The  main result of this paper is the following theorem:
 \begin{theo}\label{th2}
	Let $\varphi$ be an $m-$sh function on $\Omega$. Then for any $a\in\Omega$, the Lelong number of $\varphi$ at $a$ is given by the following limits:
    \begin{equation}\label{q1}
    \nu_\varphi(a)=2\lim_{r\to0^+}\frac{\lambda(\varphi,a,r)}{\phi_m(r)}=\frac2n \left(n-\frac{n}{m}+1\right)\lim_{r\to0^+}\frac{\Lambda(\varphi,a,r)}{\phi_m(r)}.
    \end{equation}

	 In particular, if 	$\varphi$ is bounded near $a$ then $\nu_\varphi(a)=0$.
\end{theo}
To prove this theorem we need the following lemmas where we prove some more precise results.
 \begin{lem}\label{...}
Let $a\in\Omega$,  and $ 0<r_1<r_2<d(a,\partial\Omega)$. Then
 $$\lambda(\varphi,a,r_2)-\lambda(\varphi,a,r_1)=\ds\frac12\int_{\phi_m(r_1)}^{\phi_m(r_2)}\nu_\varphi(a,\phi_m ^{-1}(t))dt.$$
\end{lem}
\begin{proof}
According to Green formula we have
$$\begin{array}{lcl}
  \lambda(\varphi,a,r)-\lambda(\varphi,a,s)&=&\ds\frac{1}{2n}\left[ \ds\int_0^r t\Lambda(\Delta \varphi,a,t)dt-\int_0^s t\Lambda(\Delta \varphi,a,t)dt\right] \\
  &=&\ds\frac{1}{2n} \ds\int_s^r t\Lambda(\Delta \varphi,a,t)dt\\
  &=&\ds\frac{1}{2n}\ds\int_s^r t\frac{n!}{\pi^n t^{2n}}\int_{\mathbb{B}(a,t)}\Delta \varphi(x)dV(x)dt\\
  &=&\ds\frac{1}{2n}\ds\int_s^r t\frac{n!}{t^{2n}}\int_{\mathbb{B}(a,t)}\frac{2}{(n-1)!}dd^c\varphi\wedge\beta^{n-1}dt\\
  &=&\ds\ds\int_s^r\frac{1}{t^{2n-1}}\ds\int_{\mathbb{B}(a,t)}dd^c\varphi\wedge \beta^{n-1}dt\\
  &=&\ds\int_s^r\frac{1}{t^{\frac{2n}{m}-1}}\nu_\varphi(a,t)dt \\
  &=&\ds\frac12\int_s^r\nu_\varphi(a,t)d\phi_m(t).
\end{array}$$
Thus, a changement of variable $t=\phi_m^{-1}(r)$ gives,
$$\lambda(\varphi,a,r_2)-\lambda(\varphi,a,r_1)=\ds\frac12\int_{\phi_m(r_1)}^{\phi_m(r_2)}\nu_\varphi(a,\phi_m ^{-1}(r))dr.$$
\end{proof}

\begin{lem}\label{l3}
Let  $a\in\Omega$. Then
$t\longmapsto\lambda(\varphi,a,r)$  is convex increasing of $t=\phi_m(r)$.
Moreover, the following limit exists in $[0,+\infty[$,
$$\nu_{\varphi}(a)=2\left.\ds\lim_{r\rightarrow0^+}\frac{\lambda(\varphi,r)}{\phi_m(r)}=2\frac{\partial^+\lambda(\varphi,r)}{\partial\phi_m(r)}\right|_{r=0}$$
\end{lem}
\begin{proof}
According to the previous lemma
$$\lambda(\varphi,a,r_2)-\lambda(\varphi,a,r_1)=\ds\frac12\int_{\phi_m(r_1)}^{\phi_m(r_2)}\nu_\varphi(a,\phi_m ^{-1}(r))dr.$$
It follows that,
$$2\frac{\partial^+\lambda(\varphi,a,r)}{\partial\phi_m(r)}=\nu_{\varphi}(a,r).$$  Since  the Lelong function $r\longmapsto\nu_\varphi(a,r)$  is increasing on $]0,r_0]$, then  the function $\phi_m(r)\mapsto\lambda(\varphi,a,r)$ is convex  increasing.\\
Furthermore, for $ 0<r_0<d(a,\partial\Omega)$ the following limit exists in $[0,+\infty[$
$$\begin{array}{lcl}
\ds 2\lim_{r\rightarrow0^+}\frac{\lambda(\varphi,a,r)}{\phi_m(r)}&=&
   \ds 2\lim_{r\rightarrow0^+}\frac{\lambda(\varphi,a,r)-\lambda(\varphi,a,r_0)}{\phi_m(r)-\phi_m(r_0)}\\
   &=&\ds 2\left.\frac{\partial^+\lambda(\varphi,a,r)}{\partial\phi_m(r)}\right|_{r=0}\\
&=&\ds\lim_{r\rightarrow0^+}\nu_\varphi(a,r).
\end{array}$$
\end{proof}
Now we can conclude the proof of Theorem \ref{th2}.
\begin{proof}  The first equality in (\ref{q1})  is proved by Lemma \ref{l3}. For the second one, using the classical formula:
$$\Lambda(\varphi,a,r)=2n\int_0^1t^{2n-1}\lambda(\varphi,a,rt)dt$$
and Lemma \ref{l3}, one can deduce that $\phi_m(r)\longmapsto\Lambda(\varphi,a,r)$ is convex increasing so the second limit in (\ref{q1}) exists. Moreover, one has
$$\begin{array}{lcl}
\ds 2\lim_{r\to0^+}\frac{\Lambda(\varphi,a,r)}{\phi(r)}&=&\ds 2n\lim_{r\to0^+}\int_0^1 2\frac{\lambda(\varphi,a,rt)}{\phi(rt)}t^{2n-\frac{2n}{m}+1}dt\\
&=&\ds 2n\nu_\varphi(a) \int_0^1t^{2n-\frac{2n}{m}+1}dt\\
&=&\ds\frac{n}{n-\frac{n}{m}+1}\nu_\varphi(a).
\end{array}$$
\end{proof}
\begin{cor}
Let $\varphi$ be an $m-$sh function on $\Omega$. Then the function $z\longmapsto \nu_\varphi(z)$ is upper semi-continuous on $\Omega$.
\end{cor}
\begin{proof}
For any $c\in\Bbb R$, we set $\Omega_c:=\{z\in\Omega;\ \nu_\varphi(z)<c\}$. To prove that $\Omega_c$ is open we claim that if $c\leq 0$ then $\Omega_c=\emptyset$. So let $c>0$ and $z\in \Omega_c$. Without loss of generality, we can assume that $\varphi<0$ on $\mathbb B(z,r_0)$ for some $0<r_0<d(z,\partial \Omega)$. Let $c'\in]\nu_\varphi(z),c[$ and $t\in]0,1[$ such that $$\frac{c'}{(1-t)^{2(n+1-\frac{n}{m})}}<c.$$
As $$\frac2n (n+1-\frac nm)\frac{\Lambda(\varphi,z,r)}{\phi(r)}$$ decreases to $\nu_\varphi(z)$, then there exists $0<r_1<r_0$ such that for every $0<r<r_1$ one has $$\frac2n (n+1-\frac nm)\frac{\Lambda(\varphi,z,r)}{\phi(r)}\leq c'.$$
Let $0<r<r_1$. Then for any $\xi\in\mathbb B(z,rt)$ one has $\mathbb B(\xi,r(1-t))\subset \mathbb B(z,r)$. Hence we obtain
$$\Lambda(\varphi,\xi,r(1-t))\geq \frac1{(1-t)^{2n}}\Lambda(\varphi,z,r)$$
Thus,
$$\begin{array}{lcl}
     \ds\frac2n (n+1-\frac nm)\frac{\Lambda(\varphi,\xi,r(1-t))}{\phi(r(1-t))}&\leq&\ds \frac{\frac2n (n+1-\frac nm)}{(1-t)^{2(n-\frac{n}{m}+1)}} \frac{\Lambda(\varphi,z,r)}{\phi(r)}\\
&\leq&\ds \frac{c'}{(1-t)^{2(n-\frac{n}{m}+1)}}.
  \end{array}
$$
We conclude that we have
$$\nu_\varphi(\xi)\leq \frac{c'}{(1-t)^{2(n-\frac{n}{m}+1)}}<c.$$
So $\xi\in \Omega_c$ for every $\xi\in\mathbb B(z,rt)$.

\end{proof}
 Since $z\mapsto\nu_{\varphi}(z)$ is upper semi continuous on $\Omega$ then it is clair that the level sets $E^m_\varphi(c):=\{z\in\Omega;\nu_{\varphi}(z)\geq c\}$ is closed.
Moreover, for any $c>0$ we have $E^m_\varphi(c)\subset\{\varphi=-\infty\}$ and its Hausdorff dimension $$\dim_{\mathcal H}(E^m_\varphi(c))\leq \frac{2n}{m}(m-1).$$
In particular, For $m=1$ we have $E^1_\varphi(c)$ is a locally finite set.\\

Indeed, let $z\in\Omega$ such that $-\infty<\varphi(z)< 0$. Then for any $0<r<d(z,\partial\Omega)$ one has $$\frac{\varphi(z)}{\phi_m(r)}\geq\frac{\lambda(\varphi,z,r)}{\phi_m(r)}\geq0.$$
Hence, when we tend $r\to 0^+$ we get $\nu_{\varphi}(z)=0$. It follows that  $z\not\in E^m_\varphi(c)$ for all $c>0$.\\
It is well known  that for a plurisubharmonic function $u$, the level set $E_{u}(c)$ is analytic  whenever $c>0$ (this result is due to Siu \cite{Si}). Thus, we can ask the following question: \\

 \begin{prob}  Let $\varphi$ be an $m-$sh function on $\Omega$. What can be said about the analyticity of level sets $E^m_\varphi(c)$ for any $c>0$?
 \end{prob}

For the maximum of $m-$sh functions on spheres/or balls we have the following proposition:
\begin{prop}\label{p3}
Let $\varphi$ be an $m-$sh function on $\Omega$ and $a$ be a point of $\Omega$. Then
$\phi_m(r)\longmapsto M(\varphi,a,r)$ is a convex increasing function on $]0,d(a,\partial\Omega)[$ where $$M(\varphi,a,r):=\sup_{\xi\in\mathbb B(a,r)}\varphi(\xi).$$
In particular, the limit $$\lim_{r\to0^+}\frac{M(\varphi,a,r)}{\phi(r)}$$ exists.
\end{prop}
\begin{proof}
Without loss of generality we can assume that $a=0\in\Omega$.\\
    Let $0<r_1<r_2<d(0,\partial\Omega)$. We consider the two following $m-$sh functions on $\Omega$:  $$u(z)=\ds\frac{\varphi(z)-M(\varphi,0,r_1)}{M(\varphi,0,r_2)-M(\varphi,0,r_1)}\quad \hbox{ and }\quad v(z)=\ds\frac{\phi_m(|z|)-\phi_m(r_1)}{\phi_m(r_2)-\phi_m(r_1)}.$$ Therefore
 \begin{itemize}
   \item for every $z\in \mathbb{B}(r_2)$ one has $u(z)\leq 1$
   \item and for every $z\in \mathbb{B}(r_1)$ one has $u(z)\leq 0$.
 \end{itemize}
 Hence,
 $$\left\{\begin{array}{lcl}
          u(z)\leq 1=v(z) &\hbox{if}\  |z|=r_2 \\
          u(z)\leq 0=v(z)& \hbox{if}\   |z|=r_1
         \end{array}\right.$$
which gives  $u(z)\leq v(z)$ for every  $z \in \partial(\mathbb{B}(r_2)\smallsetminus\mathbb{B}(r_1))$.\\ As $v$ is an $m-$sh function
and  $dd^cu^m\wedge\beta^{n-m}\geq dd^cv^m\wedge\beta^{n-m}=0$, then  thanks to the comparison principle, we have $u\leq v$ on $ \mathbb{B}(r_2)\smallsetminus\mathbb{B}(r_1)$.
Thus, for every  $z\in \mathbb{B}(r_2)\smallsetminus\mathbb{B}(r_1)$,
$$\ds\frac{\varphi(z)-M(\varphi,0,r_1)}{M(\varphi,0,r_2)-M(\varphi,0,r_1)}\leq\frac{\phi_m(|z|)-\phi_m(r_1)}{\phi_m(r_2)-\phi_m(r_1)}$$
Then, for any $r_1<r<r_2$ we have
$$\frac{M(\varphi,0,r)-M(\varphi,0,r_1)}{\phi_m(r)-\phi_m(r_1)}\leq \frac{M(\varphi,0,r_2)-M(\varphi,0,r_1)}{\phi_m(r_2)-\phi_m(r_1)}.$$
It follows that the function  $$r\longmapsto\ds\frac{M(\varphi,0,r)-M(\varphi,0,r_1)}{\phi_m(r)-\phi_m(r_1)}$$ is increasing on $]r_1,d(0,\partial\Omega))[$.
So we conclude the existence of the limit  $$\ell_{\varphi}(0):=2\ds\lim_{r\to0^+}\frac{M(\varphi,0,r)}{\phi_m(r)}.$$
\end{proof}
Claim that we have $\ell_\varphi(a)\leq \nu_\varphi(a)$ for any $a\in\Omega$ and we have equality in some particular cases of $m-$sh functions $\varphi$ on $\Omega$. Hence we can pose the following question:\\

\begin{prob}Is it true that for any $m-$sh function $\varphi$ on $\Omega$ and any $a\in\Omega$, we have
$$ \nu_\varphi(a)=2\lim_{r\to0^+}\frac{M(\varphi,a,r)}{\phi_m(r)}?$$
\end{prob}

In the following, we give an estimate to  Lelong number by the mass of $m-$sh function.
\begin{rem}
Let $a$ be a point of $\Omega$ and $\varphi$ be a function in $\mathcal{E}_m(\Omega)$. Then
$$\nu_{ \varphi}(a)\leq ((dd^c\varphi)^m\wedge\beta^{n-m}(\{a\}))^{\frac{1}{m}}$$
\end{rem}
\begin{proof}
Without loss of generality we can assume that $a=0$ and $\varphi$ belongs to $\mathcal{F}_m(\mathbb{B})$ where $\mathbb{B}=\mathbb{B}(r_0)$ is a ball. We have
$$\ds\lim_{s\rightarrow0}\int_{B(0,s)}dd^c\varphi\wedge(dd^c\widetilde{\phi}_m)^{m-1}\wedge\beta^{n-m}=\nu_{\varphi}(0).$$
Using Lemma \ref{lem1}, it follows that for $\varrho\geq1$,
$$\begin{array}{lcl}
   \nu_\varphi(0)&\leq&\ds\int_{\mathbb{B}}-\max\left(\frac{\widetilde{\phi}_m}{\varrho},-1\right)dd^c\varphi\wedge(dd^c\widetilde{\phi}_m)^{m-1} \wedge\beta^{n-m}  \\
   &  \leq&\left[\ds\int_{\mathbb{B}}-\max\left(\frac{\widetilde{\phi}_m}{\varrho},-1\right)dd^c\varphi^m\wedge\beta^{n-m}\right]^{\frac{1}{m}}\times\\
   & & \left[\ds\int_{\mathbb{B}}-\max\left(\frac{\widetilde{\phi}_m}{\varrho},-1\right)(dd^c\widetilde{\phi}_m)^m\wedge\beta^{n-m}\right]^{\frac{m-1}{m}}.
   \end{array}$$
   Since $\widetilde{\phi}_m$ is the elementary solution of the complex Hessian equation we infer,
   $$\nu_{\varphi}(0)\leq\ds\left[\int_{\mathbb{B}}-\max\left(\frac{\widetilde{\phi}_m}{\varrho},-1\right) (dd^c\varphi)^m\wedge\beta^{n-m}\right]^{\frac{1}{m}}$$

Consequently, when $\varrho$ goes to $+\infty$, we obtain
$$\nu_\varphi(0)\leq\ds(dd^c\varphi)^m\wedge\beta^{n-m}(\{0\})^{\frac{1}{m}}.$$
\end{proof}

\section{Integrability exponents of $m-$subharmonic functions}
This part is an application of previous parts where we study the integrability exponents of $m-$sh functions. This problem was posed by Blocki \cite{Bl} in 2005. We express this exponent   in terms of volume of sub-level sets of the function, then we find a relationship between it and the Lelong number of the function. In particular we determine this exponent of a $1-$sh function when its Lelong number is not equal to zero.
\begin{defn}
Let $\varphi$ be an $m-$sh function on $\Omega$ and $K$ be a compact subset of $\Omega$. The integrability exponent $\imath_K(\varphi)$ of $\varphi$ at $K$ is defined as
$$\imath_{K}(\varphi)=\sup\left\{c>0, |\varphi|^c\in L^1(\vartheta(K))\right\}.$$
For simplicity, if $K=\{x\}$ then we denote $\imath_{\{x\}}(\varphi)$ by $\imath_x(\varphi)$.
\end{defn}
 \begin{prop}
Let $\varphi$ be an $m-$sh function on $\Omega$. Then for every compact subset $K$ of $\Omega$ we have
$$\imath_K(\varphi)=\ds\inf_{x\in K}\imath_x(\varphi).$$
\end{prop}
\begin{proof}
For $x\in K$, we have$$\{c>0,|\varphi|^c\in L^1(\vartheta(K))\}\subset\{c>0,|\varphi|^c\in L^1(\vartheta(x))\}.$$
Therefore, $\imath_K(\varphi)\leq\ds \imath_{x}(\varphi)$. It follows that  $$\imath_K(\varphi)\leq \ds\inf_{x\in K}\imath_x(\varphi).$$
Conversely, let $a<\inf_{x\in K} \imath_x(\varphi)$. For any $x\in K$, let $U_x$ be a neighborhood of $x$ such that $|\varphi|^a\in L^1(U_x)$.
As $K$ is compact, then there are $x_1,...,x_p\in K$ such that $K\ds\subset\ds\cup_{j=1}^{p}U_{x_j}=U$. According to Borel-Lebesgue Lemma, we have $|\varphi|^a\in L^1(U)$. Hence $a\leqslant \imath_K(\varphi)$ so we conclude that we have $\inf_{x\in K}\imath_x(\varphi)\leqslant \imath_K(\varphi)$ .
\end{proof}
Some quite questions related  to the integrability exponents of $m-$sh function are still open, among this we can state the following:
 \begin{prob} Let $\varphi$ be an $m-$sh function on $\Omega$.
 \begin{enumerate}
   \item Are the maps  $a\longmapsto \imath_a(\varphi)$ and $\varphi\longmapsto \imath_a(\varphi)$ lower semi-continuous respectively  on $\Omega$ and on the set of locally integrable functions?
   \item Let $\mathcal I_\varphi:=\{c>0;\ |\varphi|^{c}\in L^1(\vartheta(z))\}$.  Is $\mathcal I_\varphi$ an  open set? (openness conjecture).
 \end{enumerate}

 \end{prob}
 If the map $a\longmapsto \imath_a(\varphi)$ is lower semi continuous on $\Omega$ then for every compact subset $K$ of $\Omega$ there exists $a\in K$ such that $\imath_K(\varphi)=\imath_a(\varphi)$.
\begin{lem}
Let $\varphi$, $\psi$ be two $m-$sh functions on $\Omega$ and $K$ be a compact subset of $\Omega$. If $\varphi\leq\psi$ in a neighborhood of $K$ then
 $$\imath_K(\varphi)\leq \imath_K(\psi).$$

\end{lem}
\begin{proof}
Without loss of generality we can assume that $\varphi\leq\psi\leq0$ on $\vartheta(K)$. It follows that,
$$\left\{c>0, |\varphi|^c\in L^1(\vartheta(K))\right\}\subset\left\{c>0, |\psi|^c\in L^1(\vartheta(K))\right\}$$ and the result holds.
\end{proof}
\begin{lem}\label{lemme}
Let $K$ be a compact subset of $\Omega$ and $\varphi$ be an $m-$sh function on $\Omega$, negative on a neighborhood of $K$. For any $t\in \mathbb{R}$ we set $A_{\varphi}(t)=\{z\in\Omega;\varphi(z)\leq t\}$. Then for every positive number $0<\alpha<\imath_K(\varphi)$ there exists $C_{\alpha}>0$ such that for any $t<0$ we have $$V(K\cap A_{\varphi}(t))\leq\frac{C_{\alpha}}{|t|^{\alpha}}.$$
\end{lem}
\begin{proof}
Let $0<\alpha<\imath_K(\varphi)$ and $t<0$. If $z\in A_{\varphi}(t)\cap K$ then $\frac{\varphi(z)}{t}\geq 1$.
It follows that,
$$\begin{array}{lcl}
V(A_{\varphi}(t)\cap K)&\leq&\ds\int_{A_{\varphi}(t)\cap K}\left|\frac{\varphi(z)}{t}\right|^{\alpha}dV(z)\\
                       &\leq&\ds\frac{1}{|t|^{\alpha}}\int_K|\varphi(z)|^{\alpha}dV(z)
\end{array}$$
as $\alpha<\imath_K(\varphi)$ then $$C_{\alpha}:=\int_K|\varphi(z)|^{\alpha}dV(z)<+\infty.$$
\end{proof}
\begin{theo}
Let $K$ be a compact subset of $\Omega$ and $\varphi$ be an $m-$sh function on $\Omega$, negative on a neighborhood of $K$. Then
$$\imath_K(\varphi)=\ds\sup\left\{\alpha>0; \exists\ C_{\alpha}>0, \forall\; t<0 \ V(A_{\varphi}(t)\cap K )\leq\frac{C_{\alpha}}{|t|^{\alpha}}\right\}$$
\end{theo}
A similar result for the complex singularity exponents of plurisubharmonic functions was proved by Kiselman \cite{Ki}.
\begin{proof}
In order to simplify the notations, we set $$\gamma=\sup\left\{\alpha>0; \exists\; C_{\alpha}>0,\ \forall\; t<0\ V(A_{\varphi}(t)\cap K)\leq\frac{C_{\alpha}}{|t|^{\alpha}}\right\}.$$
Thanks to Lemma \ref{lemme}, for $\alpha<\imath_K(\varphi)$ there is $C_{\alpha}>0$ such that for any $t<0$ one has
$V(A_{\varphi}(t)\cap K)\leq\frac{C_{\alpha}}{|t|^{\alpha}}$,
which means that $\gamma\geq \imath_K(\varphi)$.\\
In the other hand, let $0<\alpha_0<\gamma$, then there exists $C_{\alpha_0}>0$ such that for any $t<0$, $$V(A_{\varphi}(t)\cap K)\leq\ds\frac{C_{\alpha_0}}{|t|^{\alpha_0}}.$$
It follows that for any $0<\alpha<\alpha_0$ we have
$$\begin{array}{lcl}
\ds\int_K|\varphi(z)|^{\alpha}dV(z) &=&\ds\int_{\mathbb{R}_+} V(A_{|\varphi|^{\alpha}}(s)\cap K)ds\\
 &=&\ds\int_{\mathbb{R}_+}V(A_{\varphi}(-s^{\frac{1}{\alpha}})\cap K)ds\\
 &\leq&\ds\int_0^1V(A_{\varphi}(-s^{\frac{1}{\alpha}})\cap K)+\ds\int_1^{+\infty}\frac{C_{\alpha_0}}{s^{\frac{\alpha_0}{\alpha}}}ds \\
 &\leq& V(K)+\ds\int_1^{+\infty}\frac{C_{\alpha_0}}{s^{\frac{\alpha_0}{\alpha}}}ds<+\infty.
\end{array}$$
Thus, $\alpha\leq \imath_{K}(\varphi)$ so the result holds.
\end{proof}

The main result of this part is the following theorem:
\begin{theo}
Let $\varphi$ be an $m-$sh function on $\Omega$ and $x\in\Omega$. Then $$\imath_x(\varphi)\geq \frac{n}{n-m}.$$
Moreover, if $\nu_{\varphi}(x)>0$ then $$\imath_x(\varphi)\leq\frac{nm}{n-m}.$$
\end{theo}
The first part of this result was proved by Blocki \cite{Bl} where he has conjectured that for any $m-$sh function, the integrability exponent is greater than or equal to $\frac{nm}{n-m}$. If the conjecture is true then we conclude that we have equality in the second statement.\\
Claim that this result gives that for any $1-$sh function $\varphi$ with non vanishing Lelong number at a point $x$, the integrability exponent of $\varphi$ at $x$ is $\imath_x(\varphi)=\frac{n}{n-1}$.
\begin{proof} Without loss of generality we can assume that $x=0\in\Omega$ and $\varphi \leq0$ in a neighborhood of $0$.\\
\begin{enumerate}
\item Let $T=dd^c\varphi$ and $\chi$ be a cut-off function with support in a small ball $\mathbb{B}(r)$, equal to $1$ on $\mathbb{B}(\frac{r}{2})$. As $(dd^c\widetilde{\phi}_m)^m\wedge\beta^{n-m}=\delta_0$, for $z\in \mathbb{B}(\frac{r}{2})$
$$\begin{array}{lcl}
   \varphi(z)&=& \ds\int_{\mathbb{B}(r)}\chi(\xi)\varphi(\xi)(dd^c\widetilde{\phi}_m(z-\xi))^m\wedge\beta^{n-m}(\xi) \\
        &=&\ds\int_{\mathbb{B}(r)}\chi(\xi)\varphi(\xi)dd^c\widetilde{\phi}_m(z-\xi)\wedge(dd^c\widetilde{\phi}_m(z-\xi))^{m-1}\wedge\beta^{n-m}(\xi)\\
        &=&\ds\int_{\mathbb{B}(r)}dd^c(\chi(\xi)\varphi(\xi))\widetilde{\phi}_m(z-\xi)\wedge(dd^c\widetilde{\phi}_m(z-\xi))^{m-1}\wedge\beta^{n-m}(\xi)\\
        &=&\ds\int_{\mathbb{B}(r)}\chi(\xi)dd^c\varphi(\xi)\widetilde{\phi}_m(z-\xi)\wedge(dd^c\widetilde{\phi}_m(z-\xi))^{m-1}\wedge\beta^{n-m}(\xi)\\
        && +\ds\int_{\mathbb{B}(r)}dd^c\chi(\xi)\varphi(\xi)\widetilde{\phi}_m(z-\xi)\wedge(dd^c\widetilde{\phi}_m(z-\xi))^{m-1}\wedge\beta^{n-m}(\xi)\\
        && +\ds\int_{\mathbb{B}(r)}d\chi(\xi)\wedge d^c\varphi(\xi)\widetilde{\phi}_m(z-\xi)\wedge(dd^c\widetilde{\phi}_m(z-\xi))^{m-1}\wedge\beta^{n-m}(\xi)\\
        &&-\ds\int_{\mathbb{B}(r)}d^c\chi(\xi)\wedge d\varphi(\xi)\widetilde{\phi}_m(z-\xi)\wedge(dd^c\widetilde{\phi}_m(z-\xi))^{m-1}\wedge\beta^{n-m}(\xi)\\
        &=&\ds\int_{\mathbb{B}(r)}\chi(\xi)dd^c\varphi(\xi)\widetilde{\phi}_m(z-\xi)\wedge(dd^c\widetilde{\phi}_m(z-\xi))^{m-1}\wedge\beta^{n-m}(\xi)+ R(z) .
\end{array}$$
Where $R$ is a $\mathcal{C}^{\infty}$ function on $\mathbb{B}(r)$.
Set $$J(z)=\ds\int_{\mathbb{B}(r)}\chi(\xi)T(\xi)\wedge(dd^c\widetilde{\phi}_m(z-\xi))^{m-1}\wedge\beta^{n-m}(\xi).$$
As $$\nu_T(0,r)=\ds\int_{\mathbb{B}(r)}T(\xi)\wedge (dd^c\widetilde{\phi}_m(\xi))^{m-1}\wedge\beta^{n-m}(\xi)$$ then for $0<\delta<1$ and $r$ small enough, one has
$$\nu_T(0,\frac{r}{2})\leq J(0)\leq\nu_T(0,r)\leq \nu_T(0)+\delta.$$
By continuity, there exists $0<\varepsilon<\frac{r}{2}$ such that for every \\$z\in \mathbb{B}(\varepsilon)$ one has
$$(1-\delta)\nu_T(0,\frac{r}{2})< J(z)\leq\nu_T(0)+2\delta. $$
Fix $z\in\mathbb{B}(\varepsilon)$ and let $\mu_z$ be the probability measure defined on $\mathbb{B}(r)$ by
$$d\mu_z(\xi)=J^{-1}(z)\chi(\xi)T(\xi)\wedge(dd^c\widetilde{\phi}_m(z-\xi))^{m-1}\wedge\beta^{n-m}(\xi)$$
It follows that
$$-\varphi(z)=\ds\int_{\mathbb{B}(r)}J(z)(-\widetilde{\phi}_m(z-\xi))d\mu_z(\xi)-R(z)$$
As $R$ is $\mathcal{C}^{\infty}$ on $\mathbb{B}(r)$ then it is bounded on $\mathbb{B}(\varepsilon)$.\\
Hence, for any $p\geq1$ we have
$$\begin{array}{lcl}
  (-\varphi(z))^p&\leq&\left(\ds\int_{\mathbb{B}(r)}J(z)(-\widetilde{\phi}_m(z-\xi))d\mu_z(\xi)+C\right)^p\\
  &\leq&C^p\left(\ds-\frac{1}{C}\int_{\mathbb{B}(r)}J(z)\widetilde{\phi}_m(z-\xi)d\mu_z(\xi)+1\right)^p.
 \end{array}$$
It is easy to show that
\begin{equation}
(h+1)^p\leq h^p+\alpha_1h^{p-1}+...+\alpha_{\lfloor p\rfloor-1}h^{p-\lfloor p\rfloor+1}+\alpha(1+h), \ \ \ \forall h\geq0
\end{equation}
 for some positive constantes $\alpha_j,\alpha$ depending on $p$ where $\lfloor p\rfloor$ is the integer part of $p$.
If we set $$h(z)=\frac{-1}{C}\ds\int_{\mathbb{B}(r)}J(z)\widetilde{\phi}_m(z-\xi)d\mu_z(\xi)$$
then $$(-\varphi(z))^p\leq C^p\left(\ds\sum_{j=0}^{\lfloor p\rfloor-1}\alpha_jh^{p-j}(z)+\alpha(1+h(z))\right).$$
Thus, to prove that $(-\varphi)^p\in L^1(\mathbb{B}(\varepsilon))$, it suffices to prove that $h^s\in L^1(\mathbb{B}(\varepsilon))$ for every $s=p,\dots,p-\lfloor p\rfloor+1$.\\
Now, applying Jensen's convexity inequality with the probability measure $\mu_z$, we obtain
$$\begin{array}{lcl}
(-h(z))^s&=& \ds\frac{1}{C^s}\left(\ds\int_{\mathbb{B}(r)}J(z)(-\widetilde{\phi}_m(z-\xi))d\mu_z(\xi)\right)^s\\
         &\leq& \ds\frac{1}{C^s}\int_{\mathbb{B}(r)} J^s(z)(-\widetilde{\phi}_m(z-\xi))^sd\mu_z(\xi)\\
         &\leq&  a_s  \ds\int_{\mathbb{B}(r)}(-\widetilde{\phi}_m(z-\xi))^s\chi(\xi)T(\xi)\wedge(dd^c\widetilde{\phi}_m(z-\xi))^{m-1}\wedge\beta^{n-m}(\xi)\\
         &\leq& a_s  \ds\int_{\mathbb{B}(r)}(-\widetilde{\phi}_m(z-\xi))^sT(\xi)|z-\xi|^{\frac{-2n}{m}(m-1)}\wedge\beta^{n-1}(\xi)
\end{array}$$
where we use $a_s:=\frac{(\nu_T(0)+2\delta)^{s-1}}{C^s}$ and the fact that
 $$(dd^c\widetilde{\phi}_m(z-\xi))^{m-1}\wedge\beta^{n-m}\leq |z-\xi|^{\frac{-2n}{m}(m-1)}\beta^{n-1}(\xi).$$
 We conclude that for every $z\in\mathbb{B}(\varepsilon)$, we have
 $$\begin{array}{lcl}
(-h(z))^s&\leq& a_s\ds\int_{\mathbb{B}(r)}\ds\frac{1}{|z-\xi|^{(\frac{2n}{m}-2)s+\frac{2n}{m}(m-1)}} T(\xi)\wedge\beta^{n-1}(\xi)\\
       &\leq& a_s\ds\int_{\mathbb{B}(r)}\ds\frac{1}{|z-\xi|^{(\frac{2n}{m}-2)s+\frac{2n}{m}(m-1)}}d\sigma_T(\xi).
\end{array}$$
Hence
$$\ds\int_{\mathbb{B}(\varepsilon)}(-h(z))^pdV(z)\leq a_s\ds\int_{\mathbb{B}(\varepsilon)}\ds\int_{\mathbb{B}(r)}\frac{1}{|z-\xi|^{(\frac{2n}{m}-2)+\frac{2n}{m}(m-1)}}d\sigma_T(\xi)dV(z)$$
The Fubini theorem implies
$$\ds\int_{\mathbb{B}(\varepsilon)}(-h(z))^sdV(z)\leq a_s\ds\int_{\mathbb{B}(r)}\ds\int_{\mathbb{B}(\varepsilon)}\frac{dV(z)}{|z-\xi|^{(\frac{2n}{m}-2)s+\frac{2n}{m}(m-1)}}d\sigma_T(\xi)$$
is finite.
Indeed, for $\xi\in\mathbb{B}(\varepsilon)$ and $\eta<\varepsilon-|\xi|$ one has
$$\begin{array}{l}
\ds\int_{\mathbb{B}(\varepsilon)}\frac{dV(z)}{|z-\xi|^{(\frac{2n}{m}-2)s+\frac{2n}{m}(m-1)}}\\
=\ds\int_{\mathbb{B}(\varepsilon)\cap\mathbb{B}(\xi,\eta)}\frac{dV(z)}{|z-\xi|^{(\frac{2n}{m}-2)s+\frac{2n}{m}(m-1)}} +\ds\int_{\mathbb{B}(\varepsilon)\cap\mathbb{B}^c(\xi,\eta)}\frac{dV(z)}{|z-\xi|^{(\frac{2n}{m}-2)s+\frac{2n}{m}(m-1)}}\\
\leq C\ds\int_0^{\eta}t^{-(\frac{2n}{m}-2)s-\frac{2n}{m}(m-1)+2n-1}dt+a
\end{array}$$
the last  integral is finite  for any $s<\frac{n}{n-m}$.
\item  Assume that  $\nu_{\varphi}(0)>0$.
 Using the subharmonicity of $\varphi$ and the fact that $\varphi$ is negative on $\mathbb{B}(3r)\subset \Omega$, it is not hard to see that for any $z\in\mathbb{B}(r)\smallsetminus\{0\}$
 $$\varphi(z)\leq\Lambda(\varphi,z,2|z|)\leq\frac{1}{4^n}\Lambda(\varphi,0,|z|).$$
 Hence, for every $0<p<\imath_0(\varphi)$ we have
 $$(-\varphi(z))^p\geq\left(-\frac{1}{4^n}\Lambda(\varphi,0,|z|)\right)^p.$$
 Thanks to Theorem \ref{th2}, we have $$\frac{2}{n}(n-\frac{n}{m}+1)\frac{\Lambda(\varphi,0,s)}{\phi_m(s)}$$ decreases to $\nu_{\varphi}(0)$ as $s\searrow0$.\\
 It follows that
 $$\begin{array}{l}
 \ds\int_{\mathbb{B}(r)}(-\varphi(z))^pdV(z)\\
\geq\ds\frac{1}{4^{np}}\int_{\mathbb{B}(r)}(-\Lambda(\varphi,0,|z|))^pdV(z)\\
 \geq\ds\left(\frac{n}{2\times4^n(n+1-\frac nm)}\right)^p\frac{2\pi^n}{(n-1)!}\int_0^r\left(\frac{2(n+1-\frac nm)\Lambda(\varphi,0,t)}{n\phi_m(t)}\right)^pt^{2n-1-2p(\frac{n}{m}-1)}dt\\
 \geq\ds\left(\frac{n\nu_\varphi(0)}{2\times4^n(n+1-\frac nm)}\right)^p\frac{2\pi^n}{(n-1)!}\int_0^rt^{2n-1-2p(\frac{n}{m}-1)}dt
 \end{array}$$
     \end{enumerate}
     As $0<p<\imath_0(\varphi)$, then the first integral is finite; hence the last one is too which gives $p<\frac{nm}{n-m}$.
\end{proof}

\begin{rem}
If  the Lelong number of an $m-$sh function at a point $a\in \Omega$ vanishes  then the integrability exponent of this function at this point can be greater than $\frac{nm}{n-m}$. For example one can consider the function  $$\psi(z)=-\frac{1}{|z'|^{2\left(\frac{n-1}{m}-1\right)}}$$ where $z=(z',z_n)\in\mathbb C^{n-1}\times \mathbb C$ and $2\leq m\leq n-2$. It is simple to see that $\psi$ is an $m-$sh function with $\nu_{\psi}(0)=0$ and the integrability exponent of $\psi$ at $0$ is equal to $\frac{m(n-1)}{n-1-m}$ which is greater than $\frac{nm}{n-m}$.
\end{rem}

\section*{Acknowledgements}
Proposition \ref{p3} and Lemma \ref{l3} were proved when the first named author was visiting l'institut de math\'ematiques de Toulouse. She wishes to thank Professor Ahmed Zeriahi for the hospitality and helpful discussions.


\begin{thebibliography}
{X-XX1}
\bibitem{Bl}\textbf{Z. Blocki} Weak solutions to the complex hessian equation, Ann. Int. Fourier, Grenoble, 55,5(2005) 1735-1756.
\bibitem{De}\textbf{J-P. Demailly}, Nombres de Lelong g\'en\'erali\'es, th\'eor\`eme d'int\'egrabilit\'e et d'analyticit\'e, Acta Math. 159 (1987) 153-169.
\bibitem{De-Ph}\textbf{J-P. Demailly H.H. Pham}, A sharp lower bound for the log canonical threshold, Acta Math., 212 (2014), 1-9.
\bibitem{Dh-El}\textbf{A. Dhouib and F. Elkhadhra}, $m-$potential theory associated to a positive closed current in the class of $m-$sh functions, Comp. Var. and Ell. Equ., Vol. 61 I 7 (2016) 875-901.
\bibitem{Ha-Za}\textbf{H. Hawari and M. Zaway}, On the space of Delta $m-$subharmonic functions, Analysis Math, 42(4) (2016) 353-369.
\bibitem{Ki}\textbf{C. Kiselman}, Ensembles de sous-niveau et images inverses des fonctions plurisousharmoniques, Bull. des Sci. Math. 124 (2000), 75-92.
\bibitem{Si}\textbf{Y.T. Siu}, Analyticity of sets associated to Lelong numbers and the extension of closed positive currents, Invent. Math, 27(1974) 53-156.
\bibitem{Sk}\textbf{H. Skoda}, Prolongement des courants positifs fermes de masse finie, Inv. Math. 66, (1982) 361-376.
\bibitem{Wa-Wa}\textbf{D. Wan and W. Wang}, Complex hessian operator and Lelong number for unbounded $m-$subharmonic functions, Potential Anal. Vol. 44 (2016) 53-69.
\end{thebibliography}
 \end{document}